\documentclass[a4paper,10pt,reqno]{amsart}
\usepackage{amssymb,amsthm}
\usepackage[utf8]{inputenc}
\usepackage{aliascnt}
\usepackage{tikz}
\usetikzlibrary{matrix,arrows,positioning,backgrounds,shapes}

\definecolor{LinkColor}{rgb}{0,0,0} 
\usepackage[colorlinks=true,linkcolor=LinkColor,citecolor=LinkColor,urlcolor= LinkColor, naturalnames, hyperindex, pdfstartview=FitH, bookmarksnumbered, plainpages]{hyperref}
\usepackage[capitalise,nameinlink]{cleveref}

\newcommand{\set}[2]{\left\{\smash{#1}\left|%
    \vphantom{\smash{#1}}\vphantom{\smash{#2}}\right.\,\smash{#2}\right\}}

\newcommand{\NPG}[2]{\hyperref[NP_G]{\textup{NP(\ensuremath{#1}, \ensuremath{#2})}}}
\newcommand{\SNP}[2]{\hyperref[NPS]{\textup{SNP(\ensuremath{#1}, \ensuremath{#2})}}}
\newcommand{\NP}[3]{\hyperref[NP_subgroup]{\textup{NP\ensuremath{(#1 \leq #2, #3)}}}}

\theoremstyle{plain} 
\newtheorem{lemma}{Lemma}

\newaliascnt{prop}{lemma}
\newtheorem{proposition}[prop]{Proposition}

\aliascntresetthe{prop}
\crefname{prop}{Proposition}{Propositions}

\newaliascnt{coro}{lemma}
\newtheorem{corollary}[coro]{Corollary}

\aliascntresetthe{coro}
\crefname{coro}{Corollary}{Corollaries}

\newaliascnt{theo}{lemma}
\newtheorem{theorem}[theo]{Theorem}

\aliascntresetthe{theo}
\crefname{theo}{Theorem}{Theorems}

\theoremstyle{definition}

\newaliascnt{defn}{lemma}
\newtheorem{definition}[defn]{Definition}

\aliascntresetthe{defn}

\newaliascnt{rem}{lemma}
\newtheorem{remark}[rem]{Remark}

\aliascntresetthe{rem}

\begin{document}

\title[Subgroup Normalizer Problem for Nilpotent and Metacyclic Groups]{The Subgroup Normalizer Problem \\ for Integral Group Rings of some \\ Nilpotent and Metacyclic Groups}
\author{Andreas Bächle}

\thanks{The author is supported by the Research Foundation Flanders (FWO -- Vlaanderen).}
\date{}



\begin{abstract}
 For a group $G$ and a subgroup $H$ of $G$ this article discusses the normalizer of $H$ in the units of a group ring $RG$. We prove that $H$ is only normalized by the `obvious' units, namely products of elements of $G$ normalizing $H$ and units of $RG$ centralizing $H$, provided $H$ is cyclic. Moreover we show that the normalizers of all subgroups of certain nilpotent and metacyclic groups in the corresponding group rings are as small as possible. These classes contain all dihedral groups, all finite nilpotent groups and all finite groups with all Sylow subgroups being cyclic.
\end{abstract}
 
\maketitle

 \section{Introduction}
 
 Let throughout $G$ denote a (possibly infinite) group. Let $R$ be a commutative ring with identity element $1$. By $RG$ we denote the group ring of $G$ over $R$ and by $\textup{U}(RG)$ its group of units.
 The group $\textup{U}(RG)$ acts by conjugation on itself. Evidently, $G \leq \textup{U}(RG)$ is normalized by all elements of $G \cdot \textup{Z}(\textup{U}(RG))$. The question whether equality holds was raised by Jackowski and Marciniak \cite[3.7. Question]{JaMa} and is also one of the research problems in the fundamental book of Sehgal \cite[Problem 43]{Se}:
 \begin{equation} \textup{N}_{\textup{U}(RG)}(G) = G \cdot \textup{Z}(\textup{U}(RG)). \tag*{NP($G$, $R$)} \label{NP_G} \end{equation} 
 It turned out that answering that question in the negative by constrcuting a counterexample, was a crucial step for the construction of the counterexample for the Isomorphism Problem, the long-standing problem if a finite group is determined by its group ring over the integers; this was done by Hertweck in his PhD thesis \cite[Theorems A and B]{He1}, cf.\ also \cite[Theorems A and B]{HeAnn}. Although this example is known, there are big and important classes of groups for which \NPG{G}{\mathbb{Z}} holds: for finite groups with normal Sylow $2$-subgroups (Jackowski, Marciniak \cite{JaMa}), finite metabelian groups with abelian Sylow $2$-subgroups (Marciniak, Roggenkamp \cite{MaRo}), finite Blackburn Groups and groups having an abelian subgroup of index $2$ (Li, Parmenter, Sehgal \cite{LiPaSe}). Also for Frobenius groups (Petit Lob{\~a}o,  Polcino Milies \cite{LoPo}), finite quasi-nilpotent and finite $2$-constrained groups with $G/\operatorname{O}_2(G)$ having no chief factor of order $2$ (Hertweck, Kimmerle \cite{HeKi1}).
 Furthermore for certain infinite groups, namely locally nilpotent groups (Jespers, Juriaans, de Miranda, Rogerio \cite{JJMR}), groups with every finite normal subgroup having a normal Sylow $2$-subgroup (Hertweck \cite{HeHabil}) and for arbitrary Blackburn groups (Hertweck, Jespers \cite{HeJe}) the property \NPG{G}{\mathbb{Z}} has been verified. (In some of the results more general coefficient rings than the integers are allowed.)
 
 For subgroups $H \leq G$ of the group basis $G$, the statement corresponding to \NPG{G}{R} is
\begin{equation} \textup{N}_{\textup{U}(RG)}(H) = \textup{N}_G(H) \cdot \textup{C}_{\textup{U}(RG)}(H). \tag*{NP($H \leq G$, $R$)}\label{NP_subgroup} \end{equation} 
We say that $H \leq G$ together with the coefficient ring $R$ has the \emph{normalizer property}, if this equality holds true. Note that there are different points of view on this question: we can fix (the isomorphism type of) $H$ and see whether \NP{H}{G}{R} holds, whenever $H$ embeds into a group $G$ or we can fix a group $G$ and see whether the property is true for all its subgroups, i.e.\ whether \begin{equation*} \forall \; H \leq G \colon \qquad \textup{N}_{\textup{U}(RG)}(H) = \textup{N}_G(H) \cdot \textup{C}_{\textup{U}(RG)}(H). \tag*{SNP($G, R$)}\label{NPS}\end{equation*} In this case we say that the group $G$ together with the commutative ring $R$ has the \emph{subgroup normalizer property}.

It is interesting to investigate for which groups \NP{H}{G}{R} holds and possibly find counterexamples with $G$ having smaller order than the previous mentioned example of Hertweck (there $|G| = 2^{25}\cdot 97^2$; this group is so far the only known counterexample to \NPG{G}{\mathbb{Z}}): by doing this, we hope to obtain a better insight into the structural properties that cause such phenomena and we aim to be eventually able to construct other counterexamples to \NPG{G}{\mathbb{Z}}. This article should be seen as a first step where to find such examples.
 
The article is organized as follows: the problem is reformulated in terms of automorphisms, then, in \cref{cyclic}, we prove that \NP{H}{G}{R} holds whenever $H$ is cyclic. In \cref{nilpotent} we first prove a generalization of the famous `Coleman Lemma', which is subsequently used to show that \SNP{G}{R} holds for locally nilpotent torsion groups. In the last section we establish \SNP{G}{R} for certain metacyclic groups. This article presents some results of the author's PhD thesis \cite{BaPhD}.
 
The notation is mainly standard. Let $x$ and $y$ be elements of a group. By $x^y = y^{-1}xy$ we denote the conjugate of $x$ by $y$ and by $x^G$ the $G$-conjugacy class of $x$. Furthermore, $[x,y] = x^{-1}y^{-1}xy$ is the commutator of $x$ and $y$. For subsets $X$ and $Y$ of a group, $[X,Y]$ denotes the subgroup generated by all commutators $[x, y]$ for $x \in X$ and $y \in Y$ (if a set is a singleton we omit the set braces). For the order of an element $x$ we write $o(x)$. Moreover, $\operatorname{Aut}(G)$ denotes the group of group automorphisms of the group $G$ and $\operatorname{Inn}(G)$ its subgroup of automorphisms induced by conjugation $\operatorname{conj}(x) = (g \mapsto g^x = x^{-1}gx)$ by an element $x \in G$. Note that homomorphisms are written from the right (with exception of the augmentation of the group ring).  We always assume that a ring has an identity element. By $\mathrm{U}(R)$ we denote the group of units of a ring $R$.

\section{Automorphisms}\label{auto}
 
 The question \NP{H}{G}{R} can be restated using automorphisms.
 
 \begin{definition} Let $G$ be a group, $H \leq G$ and $R$ be a commutative ring. Set \[\operatorname{Aut}_G(H) = \set{\varphi \in \operatorname{Aut}(H)}{\varphi = \operatorname{conj}(g) \text{ for some } g \in \textup{N}_G(H)}\] and  \[\operatorname{Aut}_{RG}(H) = \set{\varphi \in \operatorname{Aut}(H)}{\varphi = \operatorname{conj}(u) \text{ for some } u \in \textup{N}_{\textup{U}(RG)}(H)},\] the groups of automorphisms of $H$ induced by elements of $G$ and $\textup{U}(RG)$, respectively. \end{definition}
  
 Clearly, for $H \leq G$ we have $\operatorname{Inn}(H) \leq \operatorname{Aut}_G(H) \leq \operatorname{Aut}_{RG}(H) \leq \operatorname{Aut}(H)$.

 \begin{lemma} For a group $G$ with subgroup $H$ and a commutative ring $R$ the following statements are equivalent
 \begin{enumerate}
  \item \NP{H}{G}{R} holds.
  \item For every $u \in \textup{N}_{\textup{U}(RG)}(H)$ there exists a $g \in \textup{N}_G(H)$ such that $[ug, H] = 1$.
  \item $\operatorname{Aut}_{RG}(H) = \operatorname{Aut}_G(H)$. \hfill $\Box$
 \end{enumerate}
 \end{lemma} 
 
\begin{remark} Assume that the outer automorphism group $\operatorname{Aut}(H)/\operatorname{Inn}(H)$ of a group $H$ is trivial. Then \NP{H}{G}{R} holds for all $H \leq G$ and all commutative rings $R$. (This happens, for example, if $H \simeq \operatorname{Aut}(S)$ for a simple, non-abelian group $S$, e.g.\ for $H \simeq S_m$ with $m \not= 6$.)\\
Note also that \SNP{G}{R} clearly holds, provided the units of $RG$ are trivial (i.e.\ $\textup{U}(RG) = \mathrm{U}(R)G$). \end{remark}

 \section{Cyclic Subgroups}\label{cyclic}

We prove that \NP{H}{G}{R} holds for all commutative rings $R$, if $H$ is cyclic (the calculations needed are similar to those which occurred in the proof of \cite[Theorem 3.1]{LoPo} and the simplified proof of \cite[17.3 Theorem]{HeHabil} suggested by I.B.S.~Passi). We need the following definitions: For $x, y \in RG$ we define $[x, y]_L = xy - yx$, the \emph{additive commutator} of $x$ and $y$. By $[RG, RG]_L$ we denote the $R$-submodule of $RG$ generated by all $[x, y]_L$, for $x, y \in RG$. Note that the map ${[-, =]_L \colon} RG \times RG \to RG \colon (x, y) \mapsto [x, y]_L$ is $R$-bilinear.
Let $\operatorname{ccl}(G)$ denote the collection of all conjugacy classes of $G$. For $C \in \operatorname{ccl}(G)$ define $\varepsilon_C \colon RG \to R \colon \sum_{g \in G} u_g g \mapsto \sum_{g \in C} u_g$, the \emph{partial augmentation map} with respect to the conjugacy class $C$.
It is straight forward to check that \[ [RG, RG]_L = \set{u \in RG}{\forall \ C \in \operatorname{ccl}(G) \colon \varepsilon_C(u) = 0}. \]

\begin{lemma}\label{prop:automorphism_from_supergroupring_leaves_ccl_invariant} Let $H \leq G$, $R$ a commutative ring, $\sigma \in \operatorname{Aut}_{RG}(H)$. Then $x\sigma \in x^G$ for every $x \in H$. 
\end{lemma}

\begin{proof} There is an element $u \in \textup{N}_{\textup{U}(RG)}(H)$ such that $\sigma = \operatorname{conj}(u)$. We have \[ x\sigma - x = u^{-1}xu - x = [u^{-1}x, u]_L \in [RG, RG]_L, \] and consequently $\varepsilon_C(x \sigma) = \varepsilon_C(x)$ for all $C \in \operatorname{ccl}(G)$. Together with $x\sigma \in H \leq G$ this implies $x\sigma \in x^G$. \end{proof}

\begin{proposition}\label{prop:NP_for_cyclic_subgroups} Let $H \leq G$. If $H$ is cyclic, then \NP H G R holds for commutative rings $R$.\end{proposition}

\begin{proof} Let $H = \langle x \rangle$ and $\sigma \in \operatorname{Aut}_{RG}(H)$, then $x\sigma \in x^G$ by \cref{prop:automorphism_from_supergroupring_leaves_ccl_invariant} and hence there is an element $g \in G$ such that $x\sigma = x\operatorname{conj}(g)$, so $\sigma = \operatorname{conj}(g) \in \operatorname{Aut}_G(H)$. \end{proof}

 \section{Nilpotent Groups}\label{nilpotent}
 
 For a rational prime $p$, the Coleman Lemma (cf.\ \cite{Col}) asserts that {\NP H G R} holds, provided $H$ is a (possibly infinite) $p$-group and $p$ is not invertible in $R$. We prove the following generalization, following the line of \cite[19.4 Lemma]{HeHabil} to boil it down to a finite group problem by considering the action on the support: For an element $u = \sum_{g \in G} u_g g \in RG$ define by $\operatorname{supp}(u) = \set{g \in G}{ u_g \not= 0}$ the \emph{support of $u$}. Note that the group ring $RG$ is equipped with an augmentation homomorphism $\varepsilon \colon RG \to R$ sending an element (expressed as a linear combination with respect to the basis $G$) to the sum of its coefficients. In particular, $\varepsilon(u) \in \textup{U}(R)$ for every unit $u \in \textup{U}(RG)$.
 
 \begin{lemma}[Coleman Lemma, relative version]  \label{prop:coleman_lemma_relative} Let $H \leq G$ and $R$ a commutative ring and $p$ a rational prime such that $p \not\in \textup{U}(R)$. Let $u \in \textup{N}_{\textup{U}(RG)}(H)$. Then there exists $P \leq H$ with $|H:P| < \infty$, $p \nmid |H:P|$ and $x \in \operatorname{supp}(u) \cap \textup{N}_G(P)$ such that $x^{-1}u \in \textup{C}_{\textup{U}(RG)}(P)$. \end{lemma}

 \begin{proof} Let $H \leq G$ and $u = \sum_{g \in G} u_g g \in \textup{N}_{\textup{U}(RG)}(H)$. For every $h \in H$ we have \begin{equation} \sum_{g \in G} u_g g = \sum_{g \in G} u_g h^{-1}gh^u, \tag{*}\label{eq:coefficients_under_operation} \end{equation} and hence \begin{equation*} g \in \operatorname{supp}(u) \quad \Leftrightarrow \quad \forall \ h \in H \colon h^{-1}gh^u \in \operatorname{supp}(u).  \end{equation*} In particular, we obtain the following (right) action of the group $H$ on the support of $u$: \begin{equation*} \begin{split} \begin{array}{ccc} \operatorname{supp}(u) \times H & \to & \operatorname{supp}(u) \\ (x, h) & \mapsto & h^{-1}xh^u. \end{array} \end{split} \end{equation*} The coefficients $u_g$ of $u$ are constant on the orbits of this action by \eqref{eq:coefficients_under_operation}. Let \[ K = \set{h \in H}{\forall \ x \in \operatorname{supp}(u) \colon h^{-1}xh^u = x} \unlhd H\] be the kernel of the action. 
Then $H/K$ is isomorphic to a subgroup of the finite group $\mathrm{Sym}(\operatorname{supp}(u))$. Let $K \leq P \leq 
H$ such that $P/K$ is a Sylow $p$-subgroup of $H/K$. The induced action of the $p$-group $P/K$ on $\operatorname{supp}(u)$ must have a fixed point $x \in \operatorname{supp}(u)$, as $\varepsilon(u) \in \textup{U}(R)$ and $p \not\in \textup{U}(R)$ by assumption on $R$. This implies $x^{-1}u \in \textup{C}_{\textup{U}(RG)}(P)$.
\end{proof}
 
The subgroup normalizer property behaves well with respect to direct products: Let $G_1$ and $G_2$ be groups and $j \in \{1, 2\}$. The natural projections ${\pi_j \colon G_1 \times G_2 \to G_j}$ give rise (by the universal property of the group ring) to ring homomorphisms $R [G_1 \times G_2] \to R G_j$, which can be restricted to homomorphisms of the unit groups, $\Pi_j \colon \textup{U}(R [G_1 \times G_2]) \to \textup{U}(R G_j)$ (the box brackets are included for better readability). With the obvious inclusion maps we have the following diagram

\begin{center}
\begin{tikzpicture}
  \matrix (m) [matrix of math nodes, row sep=4em,
  column sep=3em, text height=1.5ex, text depth=0.25ex]
  { G_1 & G_1 \times G_2 &  G_2  \\
    \textup{U}(R G_1) & \textup{U}(R [G_1 \times G_2])  & \textup{U}(R G_2) \\ };
    \path[right hook->]
       (m-1-1)  edge[transform canvas={yshift=0.5ex}] (m-1-2)
       (m-1-1)  edge (m-2-1)    
       (m-1-3)  edge (m-2-3)
       (m-1-3)  edge[transform canvas={yshift=-0.5ex}] (m-1-2)
       (m-1-2) edge (m-2-2)
       (m-2-1) edge[transform canvas={yshift=0.5ex}] (m-2-2)
       (m-2-3) edge[transform canvas={yshift=-0.5ex}] (m-2-2);
    \path[->]
       (m-1-2) edge[transform canvas={yshift=-0.5ex}] node[below] {$\pi_1$} (m-1-1)
       (m-1-2) edge[transform canvas={yshift=0.5ex}] node[above] {$\pi_2$} (m-1-3)
       (m-2-2) edge[transform canvas={yshift=-0.5ex}] node[below] {$\Pi_1$} (m-2-1) 
       (m-2-2) edge[transform canvas={yshift=0.5ex}] node[above] {$\Pi_2$} (m-2-3);
\end{tikzpicture}
\end{center}

Elements are identified with their images under the inclusion maps.

\begin{lemma}\label{prop:NP_can_be_checkd_on_factors} Let $G = G_1 \times G_2$, $R$ a commutative ring, and $H \leq G$. Assume that \NP{H\pi_1}{G\pi_1}{R} and \NP{H\pi_2}{G\pi_2}{R} hold, then also \NP{H}{G}{R} holds. \end{lemma}

\begin{proof} Let $u \in \textup{N}_{\textup{U}(RG)}(H)$. Set $H_j = H\pi_j \leq G_j$ and $u_j = u\Pi_j$. For every $x_1 \in H_1$ there exists $x_2 \in H_2$, such that $x_1x_2 \in H$. Now \[x_1^{u_1} = (x_1x_2)\pi_1^{u\Pi_1} = (x_1x_2)^u\pi_1 \in H_1, \] hence $u_1 \in \textup{N}_{\textup{U}(RG_1)}(H_1)$. Analogously, we see that $u_2 \in \textup{N}_{\textup{U}(RG_2)}(H_2)$, so we get from the assumption elements $g_j \in \textup{N}_{G_j}(H_j)$ and $z_j \in \textup{C}_{\textup{U}(RG_j)}(H_j)$ such that $u_j = g_jz_j$. To construct the corresponding units in the group ring of $G$, set $w = u_1^{-1}u_2^{-1}u$. This unit is centralizing $H$: To see this, take any $x \in H$, then \begin{align*} x^w &= ((x\pi_1)(x\pi_2))^{(u^{-1}\Pi_1)(u^{-1}\Pi_2)u} = \left((x\pi_1)^{(u^{-1}\Pi_1)} (x\pi_2)^{(u^{-1}\Pi_2)} \right)^u \\ &=  \left((x^{u^{-1}}) \pi_1 (x^{u^{-1}}) \pi_2 \right)^u = x. \end{align*} 
A similar calculation shows that the element $g$ defined as $g = g_1g_2$ acts on $H$ by conjugation like $u$ does, in particular $g \in \textup{N}_G(H)$. Obviously, $z_j \in \textup{C}_{\textup{U}(RG)}(H)$. Hence \[ u = u_1u_2u_1^{-1}u_2^{-1}u = (g_1g_2)(z_1z_2w) \in \textup{N}_G(H)\cdot\textup{C}_{\textup{U}(RG)}(H).\qedhere
\] \end{proof}
 
\begin{proposition}\label{prop:NPS_direct_products} If $G = G_1 \times G_2$ and $R$ is a commutative ring, then \[ \SNP{G}{R} \text{ holds} \quad \Leftrightarrow \quad \SNP{G_1}{R} \text{  and  } \SNP{G_2}{R} \text{ hold.} \]
\end{proposition}

\begin{proof} \begin{itemize}
 \item[``$\Rightarrow$'':] Let $H_1 \leq G_1$ and $u_1 \in \textup{N}_{\textup{U}(RG_1)}(H_1)$. Set $H = H_1 \times 1 \leq G$. Using the inclusion map, $u_1 \in \textup{N}_{\textup{U}(RG)}(H)$. By assumption there are $g \in \textup{N}_G(H)$ and $z \in \textup{C}_{\textup{U}(RG)}(H)$ such that $u_1 = gz$. Clearly $g_1 = g\pi_1 \in \textup{N}_{G_1}(H_1)$ and $z_1 = z\Pi_1 \in \textup{C}_{\textup{U}(RG_1)}(H_1)$. Hence we get $u_1 = u_1\Pi_1 = (gz)\Pi_1 = g_1z_1 \in \textup{N}_{G_1}(H_1) \cdot \textup{C}_{\textup{U}(R G_1)}(H_1)$.
 \item[``$\Leftarrow$'':] Follows from \cref{prop:NP_can_be_checkd_on_factors}.\qedhere
\end{itemize}
\end{proof}

Recall that a group is called \emph{locally nilpotent}, if every finite subset is contained in a nilpotent subgroup or, equivalently, every finitely generated subgroup is nilpotent.  For a group $G$, a \emph{$G$-adapted ring} is an integral domain of characteristic zero in which a rational prime $p$ is not invertible, whenever there exists an element of order $p$ in $G$. 

\begin{theorem}\label{prop:SNP_for_peridoic_locally_nilpotent_groups} Let $G$ be a locally nilpotent torsion group. Then \SNP{G}{R} holds for $G$-adapted rings $R$. In particular \SNP{G}{R} holds for finite nilpotent groups $G$ and $G$-adapted rings $R$.
\end{theorem}

\begin{proof} By \cite[12.1.1]{Rob} we have that $G = \mathrm{Dr}_p \; G_p$, where $p$ runs through all primes, $G_p$ denotes the unique maximal $p$-subgroup of $G$, and $\mathrm{Dr}$ stands for the restricted direct product (i.e.\ the subgroup of the direct product containing those elements for which all but finitely many coordinates are equal to the identity element). Let $H \leq G$ and $u \in \textup{N}_{\textup{U}(RG)}(H)$. Define the set $P$ of rational primes `occurring' in the support of $u$: \[P = \set{p \in \mathbb{N}}{p \textrm{ a prime and } \exists \, g \in \operatorname{supp}(u) \colon p \mid o(g)}.\] 
Now $G$ can be decomposed as $G = X \times Y$, where $X = \prod_{p \in P} G_p$ and $Y = \mathrm{Dr}_{p \not\in P} \; G_p$. It follows that $[u, Y] = 1$. Let $\kappa \colon RG \to RX$ denote the ringhomomorphism induced the projection $G \to X$. By \cref{prop:NP_can_be_checkd_on_factors} it is enough to show that $u\kappa$ acts on $H\kappa$ like an element of $G\kappa$. But this follows by induction on the finite number of primes in $P$, using \cref{prop:coleman_lemma_relative,prop:NP_can_be_checkd_on_factors}. 
\end{proof} 

\section{Metacyclic Groups and Groups of Small Order}\label{meta-cyclic}

\begin{lemma}\label{prop:NP_two_generators_big_centralizer} Let $H \leq G$ be a subgroup generated by two elements, $H = \langle c, d \rangle$, such that $c^G = c^{\textup{C}_G(d)}$, then \NP{H}{G}{R} holds for all commutative rings $R$. \end{lemma}

\begin{proof} Let $\sigma \in \operatorname{Aut}_{RG}(H)$. As $d\sigma \in d^G$ by \cref{prop:automorphism_from_supergroupring_leaves_ccl_invariant}, there exists a $g \in G$ such that $d\sigma \operatorname{conj}(g) = d$. Now $c \sigma \operatorname{conj}(g) \in c^G = c^{\textup{C}_G(d)}$ and hence there exists a $g' \in \textup{C}_G(d)$ such that $\sigma = \operatorname{conj}(gg')^{-1} \in \operatorname{Aut}_G(H)$. \end{proof}

Now we can prove the subgroup normalizer property for certain metacyclic groups:

\begin{theorem} Let\label{prop:NPS_for_some_metacyclic} $G$ be a metacyclic group admitting a short exact sequence of the form $1 \to C_m \to G \to C_n \to 1$ with $m, n \in \mathbb{Z}_{\geq 1} \cup \{\infty\}$ and let $R$ be any commutative ring. Then \SNP G R holds provided one of the following is true: \begin{enumerate}
 \item\label{case:coprime} $m, n \in \mathbb{Z}_{\geq 1}$ are coprime,
 \item\label{case:prime_quotient} $n$ is a prime.
 \item $m$ is a prime.
\end{enumerate}
\end{theorem}

\begin{proof} Let $\pi \colon G \to C_n$ be the epimorphism of $G$ onto $C_n$ in the short exact sequence in the theorem. Let $G = \langle a, b \rangle$, with $a$ a generator of the kernel of $\pi$, a normal cyclic subgroup of order $m$ of $G$. Let $H$ be a subgroup of $G$. Then there exist $c, d \in H$ with $H \cap \langle a \rangle = \langle c \rangle$ and $H = \langle c, d \rangle$.
\begin{enumerate}
 \item By assumption $(m,n) = 1$, so by the well-known Schur-Zassenhaus Theorem there exists a complement $B_1$ of $\langle a \rangle$ in $G$. Replacing $d$ by a power if necessary, we may assume that $|\langle d \rangle |$ divides $n$.
 Using this, and the fact that the complement $B_1$ is a Hall subgroup, a theorem of Philip Hall \cite[9.1.7]{Rob} implies that there exists a complement $B \simeq C_n$ of $\langle a \rangle$ in $G$ containing $\langle d \rangle$. Now \cref{prop:NP_two_generators_big_centralizer} completes the proof, as $c^G = c^{B} = c^{\textup{C}_G(d)}$.
 \item Assume that $n = p$ is a prime. If $H \leq \langle a \rangle$, then $H$ is cyclic and the statement follows from \cref{prop:NP_for_cyclic_subgroups}. Otherwise $\pi|_H$ is onto and we may assume that $d = a^k b$ for some $k \in \mathbb{Z}$. Then $c^{\textup{C}_G(d)} = c^{\langle b \rangle } = c^G$, and \cref{prop:NP_two_generators_big_centralizer} applies.
 \item Now let $m = p$ be a prime. If $H \cap \operatorname{Ker} \pi = 1$, then $H$ is cyclic and \cref{prop:NP_for_cyclic_subgroups} implies the desired statement. Otherwise we can assume $c = a$, and hence $H = \langle a, b^k \rangle$. Again, $c^{\textup{C}_G(b^k)} =c^{\langle b \rangle } = c^G$ and \cref{prop:NP_two_generators_big_centralizer} completes the proof. \qedhere
\end{enumerate}
\end{proof}

Case \eqref{case:prime_quotient} of \cref{prop:NPS_for_some_metacyclic} immediately implies:

\begin{corollary}\label{prop:NPS_for_dihedral} The subgroup normalizer property holds for dihedral groups and all commutative coefficient rings. \hfill $\Box$ \end{corollary}

\begin{corollary}\label{prop:NPS_for_square_free_order} If all Sylow subgroups of $G$ are cyclic, then the subgroup normalizer property holds for $G$ and all commutative coefficient rings. This is in particular the case for groups of square-free order.
\end{corollary}

\begin{proof} By a theorem of Hölder, Burnside, and Zassenhaus a finite group whose Sylow subgroups are all cyclic has a normal cyclic subgroup with cyclic quotient of coprime order \cite[10.1.10]{Rob}. Hence we can apply \cref{prop:NPS_for_some_metacyclic}\eqref{case:coprime}. \end{proof}

Using the idea of the proof in \cite[Example 3.1]{HeIwJeJu} we can prove ths subgroup version of \cite[Theorem 2]{LiPaSe}:

\begin{remark} Assume that the group $G$ has an abelian subgroup of index $2$ and $R$ is a commutative ring. Then \SNP{G}{R} holds. \end{remark}

Also for certain normal subgroups we can prove that the normalizer property holds:

\begin{proposition}\label{prop:NP_reduction_on_normalsubgroup_and_faithful_action} Let $H$ be a finite normal subgroup of a group $G$. If $H$ has an abelian characteristic subgroup $N$ with a complement $W$ in $H$, which acts faithfully on $N$ and $\mathrm{H}^1(W, N) = 1$, then \NP{H}{G}{R} holds, provided \NP{N}{G}{R} holds. \end{proposition}

\begin{proof} Let $\sigma \in \operatorname{Aut}_{RG}(H)$. By the assumptions we can modify $\sigma$ by an inner automorphism of $G$ and assume that $\sigma|_N = \operatorname{id}_N$. We now show that $\sigma$ also induces the identity on $H/N$. For all $w \in W$ we have $w\sigma = xv$ for some $x \in N$ and $v \in W$. We get for all $y \in N$ \[y\operatorname{conj}(w) = y^w = (y^w)\sigma = (y\sigma)^{(w\sigma)} = y^{xv} = y^v = y\operatorname{conj}(v). \] Hence $\operatorname{conj}(w)|_N = \operatorname{conj}(v)|_N$. As the action of $W$ on $N$ is faithful, we have $w = v$ and hence $w\sigma = xw$. As $W$ is a right transversal of $N$ in $H$, we get that the automorphism of $H/N$ induced by $\sigma$ is the identity. From the assumption on the cohomology together with \cite[Corollary 9.20]{Rot} we obtain $\sigma = \operatorname{conj}(z)$ for some $z \in N$. \end{proof}

In some of these cases we get a complete answer using that $\mathrm{H}^1(W, N)$ is trivial, e.g.\ if $W$ and $N$ are of coprime order (see \cite[Proposition 9.40]{Rot}). Let $H$ be a finite normal subgroup of $G$ and $R$ a commutative ring. Then \NP{H}{G}{R} holds if $H \simeq C_n \rtimes C_p$ for a prime $p$ such that $p \nmid n$ or if $H$ has an abelian normal Sylow $p$-subgroup $P$ with a complement acting faithfully on $P$ and $p$ is not invertible in $R$. For instance, \NP{H}{G}{R} holds for any commutative ring $R$, if $H \unlhd G$ and $H$ is a dihedral group of order $2m$, with $m$ odd.

\begin{remark} These results can be used to verify that \SNP{G}{\mathbb{Z}} holds for all groups $G$ of order at most $47$. \end{remark}

\bibliographystyle{amsalpha}
\bibliography{literature}

\bigskip

\begin{samepage} 
\noindent Andreas Bächle \\
Vakgroep Wiskunde,\\ Vrije Universiteit
Brussel,\\ Pleinlaan 2,\\ B-1050 Brussels\\
Belgium.\\
\href{mailto:abachle@vub.ac.be}{\texttt{ABachle@vub.ac.be}}
\end{samepage}
  
\end{document}